\documentclass[11pt,a4paper]{article}

\usepackage[utf8]{inputenc}
\usepackage[T1]{fontenc}
\usepackage{lmodern}
\usepackage{amsmath,amssymb,amsthm,mathtools}
\usepackage{mathrsfs}
\usepackage{geometry}
\usepackage{enumitem}
\usepackage{hyperref}
\usepackage{microtype}
\usepackage{comment}
\usepackage{mathtools}

\hypersetup{colorlinks=true,linkcolor=blue,citecolor=blue,urlcolor=blue}

\newtheorem{theorem}{Theorem}[section]
\newtheorem{proposition}[theorem]{Proposition}
\newtheorem{lemma}[theorem]{Lemma}
\newtheorem{corollary}[theorem]{Corollary}
\theoremstyle{definition}
\newtheorem{definition}[theorem]{Definition}
\newtheorem{assumption}[theorem]{Assumption}
\theoremstyle{remark}
\newtheorem{remark}[theorem]{Remark}

\newcommand{\R}{\mathbb{R}}
\newcommand{\Og}{\mathcal{O}_g}

\newcommand{\Vol}{\operatorname{Vol}_g}
\newcommand{\divg}{\operatorname{div}_g}
\newcommand{\gradg}{\nabla_g}

\newcommand{\plapg}{\Delta_{p,g}}

\newcommand{\normg}[1]{\lvert #1\rvert_g}
\newcommand{\ipg}[2]{\left\langle #1,#2\right\rangle_g}
\newcommand{\dd}{\mathrm{d}}

\title{\textbf{Shape Optimization and Bernoulli Free Boundary Problems\\
for the Riemannian $p$-Laplacian}}
\author{Ababacar Sadikhe DJITE$^{1,\,}$\footnote{ababacarsadikhe.djite@ucad.edu.sn} ,  Diaraf SECK$^{1,\,}$ \footnote{diaraf.seck@ucad.edu.sn}\\\\
$^{1}$ Ecole Doctorale de Math\'ematiques et Informatique U.C.A.D. Dakar,  S\'en\'egal\\
Laboratoire des Math\'ematiques de la D\'ecision et \\
d'Analyse Num\'erique, BP 16889, Dakar Fann, S\'enegal.}

\date{\today}

\begin{document}

\maketitle

\begin{abstract}
We study an exterior Bernoulli-type free boundary problem associated
with the Riemannian $p$-Laplacian on a compact Riemannian manifold.
Following the shape-optimization strategy of Ly and Seck \cite{LySeck2003}, we formulate
the corresponding volume-constrained shape optimization problem,
derive its first shape derivative, obtain the overdetermined Bernoulli
condition through a Lagrange multiplier, and establish a conditional
monotonicity result under a Riemannian contraction hypothesis. We also
illustrate the abstract framework on the round sphere, where a one-parameter family of geodesic balls gives an explicit radial
construction of the $p$-harmonic states, the Bernoulli constant, and the
boundary-flux transformation under contraction. The final
successive-approximation argument requires additional geometric and
boundary-flux stability assumptions.
\end{abstract}

{\bf Mathematical classification subject}: 49Q10, 53B20
\section{Introduction}

We consider an exterior free boundary problem of Bernoulli type
associated with the $p$-Laplacian. Such problems couple a nonlinear
elliptic equation in an unknown domain with an overdetermined
condition on its free boundary. In the Euclidean setting, the
Bernoulli problem for the $p$-Laplacian has been studied from several
different viewpoints.

Acker and Meyer established uniqueness, geometric properties, and a
successive approximation procedure for Bernoulli-type free boundary
problems associated with the $p$-Laplacian \cite{AckerMeyer1995}.
Their work also provides convergence results for the trial free
boundary method. In the convex setting, Henrot and Shahgholian
established existence of classical solutions for Bernoulli-type
problems with the $p$-Laplacian, using the method of subsolutions and
supersolutions originating in the pioneering work of Beurling for
the harmonic case \cite{Beurling1957,HenrotShahgholian2000}.

A different approach was developed by Ly and Seck
\cite{LySeck2003}. They formulate a volume-constrained shape
optimization problem associated with the exterior Bernoulli problem,
compute the shape derivative of the corresponding energy, and
derive the overdetermined boundary condition through a Lagrange
multiplier. They then establish a monotonicity property of the
associated multiplier and use it as the key ingredient in a
successive approximation procedure. This strategy yields an
existence result for the exterior Bernoulli problem for every
$1<p<\infty$, without requiring the free boundary to be convex
\cite{LySeck2003}.

The purpose of this paper is to adapt this shape-optimization
strategy to Riemannian manifolds. Let $(M,g)$ be a compact,
connected, smooth Riemannian manifold without boundary, of dimension
$n\geq2$. Let
\[
K\subset M
\]
be a compact connected domain with boundary of class $C^2$. Given a
constant $c>0$, we seek an open domain
\[
K\subset\Omega\subset M
\]
and a function
\[
u_\Omega:\Omega\setminus K\longrightarrow\R
\]
satisfying
\begin{equation}
\begin{cases}
-\plapg u_\Omega=0
&\text{in }\Omega\setminus K,\\
u_\Omega=0
&\text{on }\partial\Omega,\\
u_\Omega=1
&\text{on }\partial K,\\
-\displaystyle\frac{\partial u_\Omega}{\partial\nu_g}=c
&\text{on }\partial\Omega,
\end{cases}
\qquad 1<p<\infty,
\label{eq:main-free-boundary}
\end{equation}
where $\nu_g$ denotes the outward unit normal to $\partial\Omega$ and
\[
\frac{\partial u}{\partial\nu_g}
=
g(\gradg u,\nu_g).
\]

Problem \eqref{eq:main-free-boundary} is the natural Riemannian
analogue of the Euclidean exterior Bernoulli problem considered by
Ly and Seck \cite{LySeck2003}. The extension is not merely a change
of notation. In Euclidean space, dilations provide a canonical
comparison mechanism for nested domains. On a general Riemannian
manifold there is no corresponding global family of dilations, and
the Euclidean scaling argument cannot in general be transferred
directly to the Riemannian $p$-Laplacian.

We introduce a comparison-compatible Riemannian
contraction hypothesis, denoted by {\rm (RC)}, which replaces the
role played by Euclidean dilations in the monotonicity argument.
Under this hypothesis, together with the comparison principle and
Hopf's boundary point lemma, we obtain a conditional monotonicity
result for the Bernoulli constants associated with the optimal
domains.

The paper follows the two main ingredients of the method of
Ly--Seck \cite{LySeck2003}. First, we formulate the appropriate
volume-constrained shape optimization problem on $(M,g)$ and derive
its first shape derivative. The optimality condition is expressed
through a Lagrange multiplier and yields the overdetermined
Bernoulli condition. Second, we use the Riemannian contraction
hypothesis to establish the monotonicity needed for the successive
approximation.

An additional difficulty arises when passing to the limit in the
Riemannian setting. The approximating states are defined on moving
domains, and the normal derivatives on the free boundaries must be
controlled in order to recover the overdetermined condition in the
limit. We therefore state explicitly the compactness, regularity,
and boundary-flux stability assumptions needed in the successive
approximation argument. In particular, the final existence theorem
is a conditional Riemannian extension of the Ly--Seck construction,
rather than a claim that the free boundary problem is solvable on
every compact Riemannian manifold without further geometric
assumptions.

The compactness of the ambient manifold is useful because
admissible domains cannot escape to infinity. Nevertheless, uniform
geometric conditions on their boundaries remain necessary for the
compactness and stability arguments.

The main result of this paper is consequently a sufficient
existence criterion for the Riemannian Bernoulli problem
\eqref{eq:main-free-boundary}. Under the geometric contraction,
compactness, regularity, and boundary-flux stability assumptions
introduced below, the successive approximation converges to a
limiting domain $\Omega^\ast$ whose associated state satisfies the
Riemannian $p$-harmonic equation together with the prescribed
Bernoulli condition
\[
-\frac{\partial u_{\Omega^\ast}}{\partial\nu_g}=c
\qquad\text{on }\partial\Omega^\ast.
\]
The spherical radial construction developed later provides a
concrete geometric model in which the radial transport and boundary
flux can be computed explicitly.

\subsection{The Riemannian $p$-Laplacian}

For a sufficiently regular function $u:M\to\R$, the Riemannian
gradient $\gradg u$ is defined by
\[
g(\gradg u,X)=du(X)
\]
for every vector field $X$ on $M$. Its norm is
\[
\normg{\gradg u}
=
\sqrt{g(\gradg u,\gradg u)}.
\]

For $1<p<\infty$, the Riemannian $p$-Laplacian is
\begin{equation}
\plapg u
=
\divg\!\left(
\normg{\gradg u}^{p-2}\gradg u
\right).
\label{eq:riemannian-p-laplacian}
\end{equation}
For $p=2$, this reduces to the Laplace--Beltrami operator.

The operator \eqref{eq:riemannian-p-laplacian} is intrinsic and does
not depend on the choice of local coordinates.

\subsection{Weak Formulation}

Let $\Omega\subset M$ be a domain containing $K$. A function
\[
u\in W^{1,p}(\Omega\setminus K)
\]
is a weak solution of
\[
-\plapg u=0
\quad\text{in }\Omega\setminus K
\]
if
\begin{equation}
\int_{\Omega\setminus K}
\normg{\gradg u}^{p-2}
\ipg{\gradg u}{\gradg\varphi}
\,dV_g
=
0
\label{eq:weak-state}
\end{equation}
for every
\[
\varphi\in C_c^\infty(\Omega\setminus K).
\]
Dirichlet boundary conditions are understood in the trace sense.

\subsection{The Associated Variational Problem}

For an admissible domain $\Omega$, define
\[
\mathcal{A}(\Omega)
=
\left\{
u\in W^{1,p}(\Omega\setminus K):
\operatorname{Tr}_{\partial K}u=1,\;
\operatorname{Tr}_{\partial\Omega}u=0
\right\}.
\]
We consider the energy
\begin{equation}
J_\Omega(u)
=
\frac1p
\int_{\Omega\setminus K}
\normg{\gradg u}^{p}\,dV_g,
\qquad
u\in\mathcal{A}(\Omega).
\label{eq:state-energy}
\end{equation}
The state problem is
\begin{equation}
\inf_{u\in\mathcal{A}(\Omega)}J_\Omega(u).
\label{eq:state-minimization}
\end{equation}
Whenever the minimum is attained, its unique minimizer is denoted by
$u_\Omega$.

The associated shape functional is
\begin{equation}
J_g(\Omega)
=
J_\Omega(u_\Omega)
=
\frac1p
\int_{\Omega\setminus K}
\normg{\gradg u_\Omega}^{p}\,dV_g.
\label{eq:shape-functional}
\end{equation}

We are therefore led to the volume-constrained shape optimization
problem
\begin{equation}
\inf_{\Omega\in\Og}J_g(\Omega).
\label{eq:shape-optimization}
\end{equation}

\subsection{Admissible Class of Domains}

Let $m_0>0$. We consider open domains $\Omega\subset M$ satisfying
\[
K\subset\Omega,
\qquad
\Vol(\Omega)=m_0.
\]
In addition, we impose a uniform interior geodesic cone condition.

Let $r_0>0$ be smaller than the injectivity radius of $(M,g)$ and
let $\theta\in(0,\pi/2)$. For $x\in M$, $v\in T_xM$ with
$\normg v=1$, and $0<r\leq r_0$, define
\[
C_g(x,v,\theta,r)
=
\left\{
\exp_x(tw):
0<t<r,\;
\normg w=1,\;
g_x(w,v)>\cos\theta
\right\}.
\]
We say that $\Omega$ satisfies a uniform interior geodesic cone
condition with parameters $(r_0,\theta)$ if, for every
$x\in\partial\Omega$, there exists a unit vector $v_x\in T_xM$ such
that
\[
C_g(x,v_x,\theta,r_0)\subset\Omega.
\]

We define
\[
\Og
=
\left\{
\Omega\subset M:
\begin{array}{l}
\Omega\text{ open},\quad K\subset\Omega,\\
\Vol(\Omega)=m_0,\quad
\Omega\text{ satisfies the uniform geodesic cone condition}
\end{array}
\right\}.
\]

Since $M$ is compact, no additional outer bounded domain is required
to prevent admissible domains from escaping to infinity.

\subsection{Riemannian Hausdorff Convergence}

For compact subsets $A,B\subset M$, define
\[
d_g(x,B)=\inf_{y\in B}d_g(x,y)
\]
and
\[
\rho_g(A,B)=\sup_{x\in A}d_g(x,B).
\]
The Hausdorff distance induced by $g$ is
\[
d_H^g(A,B)
=
\max\{\rho_g(A,B),\rho_g(B,A)\}.
\]

A sequence of domains $(\Omega_j)$ is said to converge to $\Omega$ in
the Riemannian Hausdorff sense if
\begin{equation}
d_H^g(M\setminus\Omega_j,M\setminus\Omega)\longrightarrow0.
\label{eq:hausdorff-convergence}
\end{equation}
We write
\[
\Omega_j\xrightarrow{d_H^g}\Omega.
\]

\section{Shape Optimization Results}

Throughout this section, $(M,g)$ and $K$ are as above.

\subsection{Existence and Uniqueness of the State}

Let $\Omega\in\Og$. The functional \eqref{eq:state-energy} is
coercive on the affine space $\mathcal{A}(\Omega)$ after subtracting
a fixed admissible extension of the boundary data. Since
$1<p<\infty$, $W^{1,p}$ is reflexive and the map
\[
\xi\longmapsto\normg{\xi}^{p}
\]
is strictly convex. The direct method therefore gives a unique
minimizer $u_\Omega$.

\begin{proposition}[Well-posedness of the state problem]
\label{prop:state-wellposed}
Let $\Omega\in\Og$ and $1<p<\infty$. Then
\eqref{eq:state-minimization} admits a unique minimizer
$u_\Omega\in\mathcal{A}(\Omega)$. This minimizer is the unique weak
solution of
\[
\begin{cases}
-\plapg u_\Omega=0
&\text{in }\Omega\setminus K,\\
u_\Omega=1
&\text{on }\partial K,\\
u_\Omega=0
&\text{on }\partial\Omega.
\end{cases}
\]
\end{proposition}

\begin{proof}
Existence follows from the direct method of the calculus of
variations. Let $(u_j)$ be a minimizing sequence. After subtracting
a fixed function with the prescribed traces, the sequence is bounded
in $W^{1,p}_0(\Omega\setminus K)$. By reflexivity, a subsequence
converges weakly to some $u_\Omega$. The trace conditions are
preserved, and weak lower semicontinuity gives minimality.

For uniqueness, let $u_1,u_2$ be two minimizers. They satisfy the
weak equation. Subtracting the equations and testing by
$u_1-u_2$ gives
\[
\int_{\Omega\setminus K}
\left\langle
\normg{\gradg u_1}^{p-2}\gradg u_1
-
\normg{\gradg u_2}^{p-2}\gradg u_2,
\gradg u_1-\gradg u_2
\right\rangle_g
\,dV_g=0.
\]
Strict monotonicity of
$\xi\mapsto\normg{\xi}^{p-2}\xi$ implies
$\gradg u_1=\gradg u_2$ almost everywhere. Since the two functions
have the same trace, $u_1=u_2$.
\end{proof}

\subsection{Geometric Compactness of the Admissible Class}

We use the following compactness property of the admissible class.

\begin{proposition}[Hausdorff compactness]
\label{prop:hausdorff-compactness}
Let $(M,g)$ be compact and let $\Og$ satisfy the uniform interior
geodesic cone condition with fixed parameters $(r_0,\theta)$. Assume
\[
K\subset\Omega,
\qquad
\Vol(\Omega)=m_0
\]
for every $\Omega\in\Og$. Then every sequence
$(\Omega_j)\subset\Og$ admits a subsequence and an open set
$\Omega^\ast\subset M$ such that
\[
\Omega_j\xrightarrow{d_H^g}\Omega^\ast.
\]
Under the stability of the volume constraint, $\Omega^\ast\in\Og$.
\end{proposition}

\begin{remark}
The compactness of the ambient manifold prevents domains from
escaping to infinity. The uniform geodesic cone condition prevents
degeneration of the boundary geometry.
\end{remark}

\subsection{Uniform Sobolev Extension Property}

The state functions associated with a varying sequence of domains are
naturally defined on the varying perforated domains
$\Omega\setminus K$. To regard these functions as a bounded
sequence in a fixed Sobolev space, we impose the following uniform
extension hypothesis.

\begin{assumption}[Uniform Sobolev extension property]
\label{hyp:extension}
There exists a constant $C_{\rm ext}>0$, independent of
$\Omega\in\mathcal O_g$, such that, for every
$u\in W^{1,p}(\Omega\setminus K)$, there exists an extension
\[
E_\Omega u\in W^{1,p}(M)
\]
satisfying
\[
E_\Omega u=u\quad\text{a.e. in }\Omega\setminus K,
\]
and
\[
\left\|E_\Omega u\right\|_{W^{1,p}(M)}
\leq C_{\rm ext}
\left\|u\right\|_{W^{1,p}(\Omega\setminus K)}.
\]
\end{assumption}

\begin{remark}
Hypothesis~\ref{hyp:extension} is a uniform Sobolev extension
assumption on the admissible family; it is not claimed to follow from
the uniform interior geodesic cone condition alone. In the Riemannian
setting, Post--Ramos Oliv\'e--Rose~\cite{PostRamosOliveRose2023},
Theorem~1.2, construct extension operators with uniform bounds for
geometrically regular families of domains. Their result is formulated
for $H^1=W^{1,2}$ and under stronger geometric assumptions, including
uniform rolling-ball and second fundamental form bounds. Since the
present paper treats the general range $1<p<\infty$, we formulate the
corresponding $W^{1,p}$ extension property explicitly as part of the
admissible-class assumptions. 
\end{remark}

\subsection{Uniform Bounds for the State Functions}

Let
\[
\Omega_j\xrightarrow{d_H^g}\Omega^\ast,
\qquad
u_j:=u_{\Omega_j}.
\]
The minimizing property gives a uniform estimate
\begin{equation}
\int_{\Omega_j\setminus K}
\normg{\gradg u_j}^{p}\,dV_g
\leq C,
\label{eq:uniform-energy-bound}
\end{equation}
where $C$ is independent of $j$.

After applying the uniform extension operators, we obtain a bounded
sequence in $W^{1,p}(M)$. Hence, after extraction,
\[
E_{\Omega_j}u_j\rightharpoonup u^\ast
\quad\text{weakly in }W^{1,p}(M),
\]
and, by compactness of the Sobolev embedding on the compact manifold,
\[
E_{\Omega_j}u_j\to u^\ast
\quad\text{strongly in }L^p(M)
\]
and almost everywhere.

\subsection{Identification of the Limit State}

Let
\[
\varphi\in C_c^\infty(\Omega^\ast\setminus K).
\]
Hausdorff convergence implies that, for sufficiently large $j$,
\[
\operatorname{supp}\varphi\subset\Omega_j\setminus K.
\]
Thus $\varphi$ can be used as a test function in the weak equation
for $u_j$:
\[
\int_{\Omega_j\setminus K}
\normg{\gradg u_j}^{p-2}
\ipg{\gradg u_j}{\gradg\varphi}
\,dV_g=0.
\]

Set
\[
A_x(\xi)=\normg{\xi}^{p-2}\xi.
\]
The map $A_x$ is strictly monotone:
\[
\ipg{A_x(\xi)-A_x(\eta)}{\xi-\eta}
\geq0,
\]
with equality if and only if $\xi=\eta$.

A Minty argument then identifies the weak limit of
$A_x(\gradg u_j)$ with $A_x(\gradg u^\ast)$, and consequently
\[
-\plapg u^\ast=0
\quad\text{in }\Omega^\ast\setminus K.
\]
The stability of the trace conditions gives
\[
u^\ast=1\quad\text{on }\partial K,
\qquad
u^\ast=0\quad\text{on }\partial\Omega^\ast.
\]
By Proposition~\ref{prop:state-wellposed},
\[
u^\ast=u_{\Omega^\ast}.
\]

\subsection{Lower Semicontinuity of the Shape Functional}

The weak lower semicontinuity of the $p$-energy yields
\[
\frac1p
\int_{\Omega^\ast\setminus K}
\normg{\gradg u_{\Omega^\ast}}^p\,dV_g
\leq
\liminf_{j\to\infty}
\frac1p
\int_{\Omega_j\setminus K}
\normg{\gradg u_j}^p\,dV_g.
\]
Therefore,
\begin{equation}
J_g(\Omega^\ast)
\leq
\liminf_{j\to\infty}J_g(\Omega_j).
\label{eq:shape-lsc}
\end{equation}

\subsection{Existence of an Optimal Domain}

\begin{theorem}[Existence of an optimal domain]
\label{thm:optimal-domain}
Let $(M,g)$ be compact, connected and without boundary, and let
$K\subset M$ be a compact connected $C^2$ domain. Assume
$1<p<\infty$ and that $\Og$ is compact with respect to
Riemannian Hausdorff convergence. Then there exists
\[
\Omega^\ast\in\Og
\]
such that
\[
J_g(\Omega^\ast)
=
\min_{\Omega\in\Og}J_g(\Omega).
\]
\end{theorem}

\begin{proof}
Let $(\Omega_j)\subset\Og$ be a minimizing sequence. By
Proposition~\ref{prop:hausdorff-compactness}, after extraction,
\[
\Omega_j\xrightarrow{d_H^g}\Omega^\ast
\]
for some $\Omega^\ast\in\Og$. The uniform estimates and the
identification of the limit state give
\[
u_j\rightharpoonup u_{\Omega^\ast}.
\]
By \eqref{eq:shape-lsc},
\[
J_g(\Omega^\ast)
\leq
\liminf_{j\to\infty}J_g(\Omega_j).
\]
Since $(\Omega_j)$ is minimizing, the right-hand side equals
$\inf_{\Omega\in\Og}J_g(\Omega)$. The reverse inequality follows
from the definition of the infimum.
\end{proof}

\subsection{Shape Derivative of the Energy Functional}

Let $V\in C_c^1(M,TM)$ be a vector field vanishing in a neighbourhood
of $K$, and let $(\Phi_t)$ be its local flow:
\[
\frac{\dd}{\dd t}\Phi_t(x)
=
V(\Phi_t(x)),
\qquad
\Phi_0(x)=x.
\]
Define
\[
\Omega_t=\Phi_t(\Omega).
\]
Since $V$ vanishes near $K$, the obstacle remains fixed.

Let $u_t=u_{\Omega_t}$. The material derivative is
\[
\dot u
=
\left.
\frac{\dd}{\dd t}
(u_t\circ\Phi_t)
\right|_{t=0},
\]
and the shape derivative is
\[
u'
=
\dot u-g(\gradg u_\Omega,V).
\]
Differentiating the boundary condition
$u_t=0$ on $\partial\Omega_t$ gives
\[
u'=-g(\gradg u_\Omega,V)
\quad\text{on }\partial\Omega.
\]
Since $u_\Omega=0$ on $\partial\Omega$,
\[
\gradg u_\Omega
=
\frac{\partial u_\Omega}{\partial\nu_g}\nu_g,
\]
and therefore
\begin{equation}
u'
=
-\frac{\partial u_\Omega}{\partial\nu_g}
g(V,\nu_g)
\quad\text{on }\partial\Omega.
\label{eq:shape-derivative-boundary}
\end{equation}

\begin{proposition}[Riemannian shape derivative]
\label{prop:shape-derivative}
Let $\Omega\in\Og$ have sufficiently regular boundary and let
$V\in C_c^1(M,TM)$ vanish in a neighbourhood of $K$. Then
\begin{equation}
J_g'(\Omega)(V)
=
-\frac{p-1}{p}
\int_{\partial\Omega}
\normg{\gradg u_\Omega}^{p}
g(V,\nu_g)\,dS_g.
\label{eq:hadamard}
\end{equation}
Equivalently,
\begin{equation}
J_g'(\Omega)(V)
=
-\frac{p-1}{p}
\int_{\partial\Omega}
\left|
\frac{\partial u_\Omega}{\partial\nu_g}
\right|^p
g(V,\nu_g)\,dS_g.
\label{eq:hadamard-normal}
\end{equation}
\end{proposition}

\begin{proof}
Apply the Riemannian transport formula to
\[
J_g(\Omega)
=
\frac1p
\int_{\Omega\setminus K}
\normg{\gradg u_\Omega}^{p}\,dV_g.
\]
The derivative with respect to the state vanishes because
$u_\Omega$ satisfies the Euler--Lagrange equation
$-\plapg u_\Omega=0$. Since $u_\Omega=0$ on the moving boundary,
the boundary integrand reduces to
\[
\frac1p\normg{\gradg u_\Omega}^{p}
-
\normg{\gradg u_\Omega}^{p}
=
-\frac{p-1}{p}
\normg{\gradg u_\Omega}^{p}.
\]
This gives \eqref{eq:hadamard}.
\end{proof}

\subsection{Optimality Condition Under a Volume Constraint}

We impose
\[
\Vol(\Omega)=m_0.
\]
For $\Omega_t=\Phi_t(\Omega)$,
\begin{equation}
\left.
\frac{\dd}{\dd t}\Vol(\Omega_t)
\right|_{t=0}
=
\int_{\partial\Omega}
g(V,\nu_g)\,dS_g.
\label{eq:volume-variation}
\end{equation}

Let $\Omega$ be a minimizer of $J_g$ under the volume constraint.
By the Lagrange multiplier principle, there exists
$\lambda_\Omega\in\R$ such that
\[
J_g'(\Omega)(V)
-
\lambda_\Omega
\left.
\frac{\dd}{\dd t}\Vol(\Omega_t)
\right|_{t=0}
=0
\]
for every admissible deformation field $V$.

Using \eqref{eq:hadamard} and \eqref{eq:volume-variation},
\[
\int_{\partial\Omega}
\left[
-\frac{p-1}{p}
\normg{\gradg u_\Omega}^{p}
-\lambda_\Omega
\right]
g(V,\nu_g)\,dS_g
=0.
\]
Since the normal component of $V$ can be prescribed arbitrarily,
\[
-\frac{p-1}{p}
\normg{\gradg u_\Omega}^{p}
-\lambda_\Omega
=0
\quad\text{on }\partial\Omega.
\]
Hence
\begin{equation}
\normg{\gradg u_\Omega}^{p}
=
-\frac{p}{p-1}\lambda_\Omega
\quad\text{on }\partial\Omega.
\label{eq:lambda-condition}
\end{equation}
In particular,
\[
\lambda_\Omega<0.
\]
Define
\begin{equation}
c_\Omega
=
\left(
-\frac{p}{p-1}\lambda_\Omega
\right)^{1/p}.
\label{eq:c-definition}
\end{equation}
Then
\[
\normg{\gradg u_\Omega}
=
c_\Omega
\quad\text{on }\partial\Omega.
\]
Since $u_\Omega>0$ in $\Omega\setminus K$ and
$u_\Omega=0$ on $\partial\Omega$,
\[
\frac{\partial u_\Omega}{\partial\nu_g}<0
\quad\text{on }\partial\Omega.
\]
Therefore
\begin{equation}
-\frac{\partial u_\Omega}{\partial\nu_g}
=
c_\Omega
=
\left(
-\frac{p}{p-1}\lambda_\Omega
\right)^{1/p}
\quad\text{on }\partial\Omega.
\label{eq:bernoulli-optimality}
\end{equation}

Thus every sufficiently regular optimal domain satisfies the
overdetermined Riemannian Bernoulli problem
\begin{equation}
\begin{cases}
-\plapg u_\Omega=0
&\text{in }\Omega\setminus K,\\
u_\Omega=0
&\text{on }\partial\Omega,\\
u_\Omega=1
&\text{on }\partial K,\\
-\displaystyle\frac{\partial u_\Omega}{\partial\nu_g}
=c_\Omega
&\text{on }\partial\Omega.
\end{cases}
\label{eq:optimal-overdetermined}
\end{equation}

\subsection{Riemannian Contraction }

The monotonicity argument in the Euclidean setting relies on
dilations. Such dilations are not available on a general Riemannian
manifold. We introduce the following additional
comparison-compatible hypothesis.

\medskip
\noindent
\textbf{(RC) Riemannian contraction hypothesis.}

Let
\[
\Omega_1,\Omega_2\in\Og,
\qquad
K\subset\Omega_1\subsetneq\Omega_2.
\]
We assume that there exist an admissible domain
$\Omega_\Psi\subset\Omega_1$, a point
\[
x_0\in\partial\Omega_\Psi\cap\partial\Omega_1,
\]
and a $C^2$ diffeomorphism
\[
\Psi:
\overline{\Omega_2\setminus K}
\longrightarrow
\overline{\Omega_\Psi\setminus K}
\]
such that
\[
\Psi(\partial K)=\partial K,
\]
and
\[
\nu_{\Omega_\Psi}(x_0)
=
\nu_{\Omega_1}(x_0).
\]

Let $u_{\Omega_2}$ be the state solution in $\Omega_2$ and define
\[
v=u_{\Omega_2}\circ\Psi^{-1}
\quad\text{in }\Omega_\Psi\setminus K.
\]
We assume that $v$ is a weak subsolution, namely
\begin{equation}
\int_{\Omega_\Psi\setminus K}
\normg{\gradg v}^{p-2}
\ipg{\gradg v}{\gradg\varphi}
\,dV_g
\leq0
\label{eq:RC-subsolution}
\end{equation}
for every nonnegative
\[
\varphi\in C_c^\infty(\Omega_\Psi\setminus K).
\]
Moreover,
\[
v=1\quad\text{on }\partial K,
\qquad
v=0\quad\text{on }\partial\Omega_\Psi.
\]
Writing
\[
y_0=\Psi^{-1}(x_0),
\]
we assume that there exists
\[
\kappa_\Psi>1
\]
such that
\begin{equation}
-\frac{\partial v}{\partial\nu_g}(x_0)
\geq
\kappa_\Psi
\left(
-\frac{\partial u_{\Omega_2}}{\partial\nu_g}(y_0)
\right).
\label{eq:RC-flux}
\end{equation}

The constant $\kappa_\Psi$ may depend on $\Psi$, but is independent
of the particular solution $u_{\Omega_2}$.

In the Euclidean setting, the homothety $\Psi_t(x)=tx$ gives
$\kappa_\Psi=1/t>1$. The hypothesis (RC) is an additional geometric
and analytic assumption and is not automatic on an arbitrary
Riemannian manifold.

\subsection{Monotonicity of the Lagrange Multiplier}

\begin{theorem}[Monotonicity of the Lagrange multiplier]
\label{thm:monotonicity-lambda}
Let $(M,g)$ be a Riemannian manifold and let $K\subset M$ be a fixed
compact obstacle. Let
\[
\Omega_1,\Omega_2\in\Og
\]
be admissible $C^2$ domains satisfying
\[
K\subset\Omega_1\subsetneq\Omega_2.
\]
Assume that \textnormal{(RC)} holds. Let $u_{\Omega_i}$ be the
corresponding state solutions and let $\lambda_{\Omega_i}$ be the
associated Lagrange multipliers. Then
\[
\lambda_{\Omega_1}<\lambda_{\Omega_2}.
\]
Equivalently,
\[
c_{\Omega_1}>c_{\Omega_2}.
\]
\end{theorem}

\begin{proof}
Set
\[
u_1=u_{\Omega_1},
\qquad
u_2=u_{\Omega_2}.
\]
By (RC), there exist $\Omega_\Psi$, $\Psi$ and $x_0$ as above.
Set
\[
v=u_2\circ\Psi^{-1}.
\]
By (RC), $v$ is a weak subsolution in
$\Omega_\Psi\setminus K$ and
\[
v=1\quad\text{on }\partial K,
\qquad
v=0\quad\text{on }\partial\Omega_\Psi.
\]

On the other hand,
\[
-\plapg u_1=0
\quad\text{in }\Omega_1\setminus K,
\]
with
\[
u_1=1\quad\text{on }\partial K,
\qquad
u_1=0\quad\text{on }\partial\Omega_1.
\]
Since $\Omega_\Psi\subset\Omega_1$, we may restrict $u_1$ to
$\Omega_\Psi\setminus K$.

Set
\[
w=(v-u_1)^+.
\]
The boundary values imply
\[
w=0
\quad\text{on }\partial(\Omega_\Psi\setminus K),
\]
so
\[
w\in W_0^{1,p}(\Omega_\Psi\setminus K).
\]
Using $w$ as a test function and subtracting the weak relations gives
\[
\int_{\{v>u_1\}}
\ipg{
\normg{\gradg v}^{p-2}\gradg v
-
\normg{\gradg u_1}^{p-2}\gradg u_1
}{
\gradg v-\gradg u_1
}
\,dV_g
\leq0.
\]
Strict monotonicity of
\[
\xi\mapsto\normg{\xi}^{p-2}\xi
\]
implies
\[
\gradg(v-u_1)^+=0
\quad\text{a.e.}
\]
Since $(v-u_1)^+$ has zero trace,
\[
v\leq u_1
\quad\text{in }\Omega_\Psi\setminus K.
\]

At $x_0$,
\[
v(x_0)=u_1(x_0)=0.
\]
Assuming $v\not\equiv u_1$, the strong comparison principle and
the Hopf boundary point lemma give
\[
-\frac{\partial u_1}{\partial\nu_g}(x_0)
>
-\frac{\partial v}{\partial\nu_g}(x_0).
\]
The equality of the outward normals at $x_0$ in (RC) allows the same
normal to be used in this comparison.

For the stationary domains,
\[
-\frac{\partial u_1}{\partial\nu_g}
=c_{\Omega_1}
\quad\text{on }\partial\Omega_1,
\]
and
\[
-\frac{\partial u_2}{\partial\nu_g}
=c_{\Omega_2}
\quad\text{on }\partial\Omega_2.
\]
With $y_0=\Psi^{-1}(x_0)$, the flux condition in (RC) gives
\[
-\frac{\partial v}{\partial\nu_g}(x_0)
\geq
\kappa_\Psi c_{\Omega_2}.
\]
Consequently,
\[
c_{\Omega_1}
>
\kappa_\Psi c_{\Omega_2}
>
c_{\Omega_2}.
\]
Thus
\[
c_{\Omega_1}>c_{\Omega_2}.
\]
Since
\[
\lambda_\Omega
=
-\frac{p-1}{p}c_\Omega^p,
\]
the strict decrease of
$c\mapsto-\frac{p-1}{p}c^p$ yields
\[
\lambda_{\Omega_1}<\lambda_{\Omega_2}.
\]
\end{proof}

\begin{remark}[Euclidean case]
In Euclidean space, $\Psi_t(x)=tx$ with $0<t<1$ and
\[
-\partial_\nu u_t
=
\frac1t(-\partial_\nu u_{\Omega_2}),
\]
so $\kappa_\Psi=1/t>1$.
\end{remark}
\subsection{Boundary-Flux Stability}

Weak convergence of the state solutions is not sufficient by itself
to pass to the limit in the overdetermined boundary condition.
We therefore impose the following additional assumption.

\medskip
\noindent
\textbf{(SC) Boundary-flux stability.}

Let
\[
\Omega_j\in\Og,
\qquad
\Omega_j\xrightarrow{d_H^g}\Omega^\ast,
\qquad
\Omega^\ast\in\Og.
\]
Let
\[
u_j=u_{\Omega_j},
\qquad
u^\ast=u_{\Omega^\ast}.
\]
Assume that there exist boundary identifications
\[
\Phi_j:\partial\Omega^\ast\longrightarrow\partial\Omega_j
\]
such that
\[
\Phi_j\longrightarrow
\operatorname{Id}_{\partial\Omega^\ast}
\]
in the relevant boundary topology, and
\begin{equation}
\sup_{x\in\partial\Omega^\ast}
\left|
\frac{\partial u_j}{\partial\nu_g}
\bigl(\Phi_j(x)\bigr)
-
\frac{\partial u^\ast}{\partial\nu_g}(x)
\right|
\longrightarrow0.
\label{eq:SC}
\end{equation}

\begin{remark}
Hypothesis (SC) is used only for the passage to the limit in the
overdetermined boundary condition. The identification of the weak
limit of the state solutions with $u_{\Omega^\ast}$ is treated
separately through the stability of the state equation and the
strict monotonicity of the $p$-Laplacian.
\end{remark}


\subsection{Reference Domains for the Successive Approximation}

The successive approximation procedure of Ly--Seck relies in the
Euclidean setting on a nested family of balls. Such a family is not
canonically available on a general Riemannian manifold. We therefore
introduce the following additional hypothesis.

\medskip
\noindent
\textbf{(RA) Riemannian reference-family hypothesis.}

We assume that there exist an interval
\[
(R_-,R_+)\subset(0,\infty)
\]
and a family of admissible $C^2$ domains
\[
\{B_R\}_{R\in(R_-,R_+)}\subset\mathcal{O}_g
\]
such that $K\subset B_R$ for every $R$ and
\[
R_1<R_2\Longrightarrow\overline{B_{R_1}}\subset B_{R_2}.
\]
We further assume Hausdorff continuity in $R$. For every $R$, let
$U_R=u_{B_R}$ and assume that its outward normal derivative
is constant on $\partial B_R$. Define
\[
q(R):=-\frac{\partial U_R}{\partial\nu_g}.
\]
Assume that $q$ is continuous and strictly decreasing and that
\begin{equation}
\lim_{R\downarrow R_-}q(R)>c,\qquad
\lim_{R\uparrow R_+}q(R)<c.
\label{eq:reference-family-crossing}
\end{equation}
Thus there is a unique $R^\ast$ such that $q(R^\ast)=c$.

\subsection{Successive Approximation}

Fix $\delta>0$ and choose $R_N$ such that
\begin{equation}
|q(R_N)-c|<\delta.
\label{eq:initial-flux-approximation}
\end{equation}
Choose $V_0$ satisfying
\begin{equation}
0<V_0<\operatorname{Vol}_g(B_{R_N}),
\end{equation}
and define
\begin{equation}
\mathcal{O}_0=\{\Omega\in\mathcal{O}_g:\Omega\subset B_{R_N},\;\operatorname{Vol}_g(\Omega)=V_0\}.
\label{eq:O0}
\end{equation}
Assuming this class is nonempty, choose an optimizer $\Omega_0$.
Set $u_0=u_{\Omega_0}$ and
\[
c_0=\left(-\frac{p}{p-1}\lambda_{\Omega_0}\right)^{1/p}.
\]
Then, whenever the boundary is sufficiently regular,
$-\partial_{\nu_g}u_0=c_0$ on $\partial\Omega_0$.

Suppose recursively that $\Omega_j$ is constructed. Choose
$V_{j+1}<\operatorname{Vol}_g(\Omega_j)$ and set
\begin{equation}
\mathcal{O}_{j+1}=\{\Omega\in\mathcal{O}_g:\Omega\subset\Omega_j,\;\operatorname{Vol}_g(\Omega)=V_{j+1}\}.
\label{eq:recursive-admissible-class}
\end{equation}
Assuming nonemptiness and compactness, let $\Omega_{j+1}$ minimize $J_g$
on this class and define
\[
c_{j+1}=\left(-\frac{p}{p-1}\lambda_{\Omega_{j+1}}\right)^{1/p}.
\]
Then
\begin{equation}
-\frac{\partial u_{j+1}}{\partial\nu_g}=c_{j+1}\quad\text{on }\partial\Omega_{j+1}.
\label{eq:recursive-bernoulli}
\end{equation}

\begin{proposition}[Monotonicity of the successive constants]
\label{prop:successive-c-monotonicity}
If $\Omega_{j+1}\subsetneq\Omega_j$, then
\[
\lambda_{\Omega_{j+1}}>\lambda_{\Omega_j},\qquad c_{j+1}<c_j.
\]
\end{proposition}

\begin{proof}
Since
\[
\Omega_{j+1}\subsetneq\Omega_j,
\]
the Riemannian contraction hypothesis {\rm (RC)} can be applied to the
pair $(\Omega_{j+1},\Omega_j)$. In particular, Theorem~\ref{thm:monotonicity-lambda}
gives
\[
\lambda_{\Omega_{j+1}}>\lambda_{\Omega_j}.
\]
Recall that both Lagrange multipliers are strictly negative and that the
Bernoulli constants are defined by
\[
c_j
=
\left(
-\frac{p}{p-1}\lambda_{\Omega_j}
\right)^{1/p},
\qquad
c_{j+1}
=
\left(
-\frac{p}{p-1}\lambda_{\Omega_{j+1}}
\right)^{1/p}.
\]
To determine the direction of the corresponding inequality for the
Bernoulli constants, consider
\[
F:(-\infty,0)\longrightarrow(0,\infty),
\qquad
F(\lambda)
=
\left(
-\frac{p}{p-1}\lambda
\right)^{1/p}.
\]
A direct differentiation gives
\[
F'(\lambda)
=
-\frac1{p}
\left(\frac{p}{p-1}\right)^{1/p}
(-\lambda)^{1/p-1}<0,
\qquad \lambda<0.
\]
Thus $F$ is strictly decreasing on $(-\infty,0)$. Therefore
\[
\lambda_{\Omega_{j+1}}>\lambda_{\Omega_j}
\quad\Longrightarrow\quad
F(\lambda_{\Omega_{j+1}})
<
F(\lambda_{\Omega_j}),
\]
that is,
\[
c_{j+1}<c_j.
\]
Equivalently, using
\[
\lambda_{\Omega_j}
=
-\frac{p-1}{p}c_j^p,
\]
the strict decrease of $c_j$ is exactly equivalent to the strict increase
of $\lambda_{\Omega_j}$. Hence
\[
\lambda_{\Omega_j}<\lambda_{\Omega_{j+1}}<0,
\qquad
c_{j+1}<c_j.
\]
In particular, every strict contraction of the admissible domain produces
a strict decrease of the Bernoulli constant.
\end{proof}

\subsection{Limit of the Successive Approximation}
\label{sec:limit-successive-approximation}

We now analyze the limiting behavior of the successive construction.
The argument separates the compactness and state-stability issues from
the identification of the limiting Bernoulli constant.

Suppose that the construction produces a nested sequence
\[
\Omega_0\supsetneq\Omega_1\supsetneq\Omega_2\supsetneq\cdots
\]
and, by the monotonicity theorem, a strictly decreasing sequence
\[
c_0>c_1>c_2>\cdots.
\]
Assume that every finite level remains strictly above the prescribed
constant:
\begin{equation}
c_j>c
\qquad\text{for every }j.
\label{eq:successive-c-bounded}
\end{equation}
Then there exists $\ell\ge c$ such that
\begin{equation}
c_j\downarrow\ell.
\label{eq:c-limit}
\end{equation}

The nestedness of the domains and the compactness assumptions on
$\mathcal O_g$ yield, after passage to a subsequence if necessary,
\[
\Omega_j\xrightarrow{d_H^g}\Omega^\ast
\]
for some $\Omega^\ast\in\mathcal O_g$. State stability identifies the
limit of the corresponding states with $u^\ast=u_{\Omega^\ast}$.
Thus
\begin{equation}
\begin{cases}
-\Delta_{p,g}u^\ast=0
&\text{in }\Omega^\ast\setminus K,\\
u^\ast=1
&\text{on }\partial K,\\
u^\ast=0
&\text{on }\partial\Omega^\ast.
\end{cases}
\label{eq:limit-state}
\end{equation}

By (SC), the normal derivatives are stable under the boundary
identifications. Since
\[
-\partial_{\nu_g}u_j=c_j
\qquad\text{on }\partial\Omega_j,
\]
we obtain
\begin{equation}
-\partial_{\nu_g}u^\ast=\ell
\qquad\text{on }\partial\Omega^\ast.
\label{eq:limit-bernoulli-level}
\end{equation}
Hence the limiting pair solves the Bernoulli problem with constant
$\ell$. The stopping argument below identifies $\ell$ with the
prescribed constant $c$.

\begin{proposition}[Limit problem]
\label{prop:limit-problem}
Under the compactness, state-stability and boundary-flux stability
assumptions, if
\[
\Omega_j\to\Omega^\ast,
\qquad
c_j\downarrow\ell,
\]
then the limiting pair satisfies
\begin{equation}
\begin{cases}
-\Delta_{p,g}u^\ast=0
&\text{in }\Omega^\ast\setminus K,\\
u^\ast=1
&\text{on }\partial K,\\
u^\ast=0
&\text{on }\partial\Omega^\ast,\\
-\partial_{\nu_g}u^\ast=\ell
&\text{on }\partial\Omega^\ast.
\end{cases}
\end{equation}
\end{proposition}

\begin{proof}
The proof consists of two logically distinct parts: first, we identify the
limit of the state equation; second, we pass to the limit in the
overdetermined boundary condition.

\medskip
\noindent\textbf{Step 1: convergence of the domains and of the states.}
By assumption,
\[
\Omega_j\xrightarrow{d_H^g}\Omega^\ast,
\qquad
\Omega^\ast\in\mathcal O_g.
\]
The compactness and state-stability hypotheses for the admissible class
allow us, after extraction of a subsequence when necessary, to obtain
\[
u_j\rightharpoonup u^\ast
\qquad\text{in }W^{1,p}(M),
\]
with the limit identified by the stability of the state equation as
\[
u^\ast=u_{\Omega^\ast}.
\]
Consequently,
\[
-\Delta_{p,g}u^\ast=0
\qquad\text{in }\Omega^\ast\setminus K.
\]

\medskip
\noindent\textbf{Step 2: passage to the Dirichlet boundary conditions.}
For every $j$,
\[
u_j=1\quad\text{on }\partial K,
\qquad
u_j=0\quad\text{on }\partial\Omega_j.
\]
The inner boundary $\partial K$ is fixed, so the trace stability gives
\[
u^\ast=1\quad\text{on }\partial K.
\]
The convergence of the outer boundaries in the topology of the
admissible class, together with the corresponding trace stability, gives
\[
u^\ast=0\quad\text{on }\partial\Omega^\ast.
\]
Thus $u^\ast$ is precisely the state associated with $\Omega^\ast$:
\[
\begin{cases}
-\Delta_{p,g}u^\ast=0
&\text{in }\Omega^\ast\setminus K,\\
u^\ast=1
&\text{on }\partial K,\\
u^\ast=0
&\text{on }\partial\Omega^\ast.
\end{cases}
\]

\medskip
\noindent\textbf{Step 3: passage to the Bernoulli condition.}
For every $j$, the first-order optimality condition gives
\[
-\frac{\partial u_j}{\partial\nu_g}=c_j
\qquad\text{on }\partial\Omega_j.
\]
By hypothesis,
\[
c_j\downarrow\ell.
\]
The difficulty is that weak convergence in $W^{1,p}$ alone does not
permit one to pass directly to normal derivatives on moving
boundaries. This is precisely the reason for introducing the
boundary-flux stability hypothesis {\rm (SC)}.

Let
\[
\Phi_j:\partial\Omega^\ast\longrightarrow\partial\Omega_j
\]
be the boundary identifications furnished by {\rm (SC)}. Then
\[
\sup_{x\in\partial\Omega^\ast}
\left|
\frac{\partial u_j}{\partial\nu_g}(\Phi_j(x))
-
\frac{\partial u^\ast}{\partial\nu_g}(x)
\right|
\longrightarrow0.
\]
Since
\[
\frac{\partial u_j}{\partial\nu_g}=-c_j
\qquad\text{on }\partial\Omega_j,
\]
we have, uniformly for $x\in\partial\Omega^\ast$,
\[
\left|
-c_j-\frac{\partial u^\ast}{\partial\nu_g}(x)
\right|
\longrightarrow0.
\]
Passing to the limit and using $c_j\to\ell$ yields
\[
\frac{\partial u^\ast}{\partial\nu_g}=-\ell
\qquad\text{on }\partial\Omega^\ast,
\]
or equivalently
\[
-\frac{\partial u^\ast}{\partial\nu_g}
=\ell
\qquad\text{on }\partial\Omega^\ast.
\]
Therefore $(\Omega^\ast,u^\ast)$ satisfies the overdetermined Bernoulli
problem with Bernoulli constant $\ell$.
\end{proof}

\subsection{Asymptotically Conformal Contractions and Flux Stability}
\label{sec:asymptotically-conformal-contractions}

We make precise the geometric condition under which the flux of the
transported reference state approaches the reference flux. This replaces
an abstract assumption on the limiting Bernoulli level by a condition on
the first jet of the Riemannian contractions.

We distinguish carefully between the index $j$, which labels the
successive approximation scheme, and the parameter $\varepsilon$, which
will be used below for a one-parameter family of contractions in the
spherical model. Thus, at the $j$-th step, $\Psi_j$ denotes the
diffeomorphism realizing the contraction from the current domain to the
next one,
\[
\Psi_j:\Omega_j\longrightarrow\Omega_{j+1},
\qquad
\Psi_j(\Omega_j)=\Omega_{j+1}.
\]
The adjective ``contraction'' refers here to the inclusion
$\Omega_{j+1}\subset\Omega_j$; it does not require
$\|d\Psi_j\|_g<1$ pointwise.

For a contraction $\Psi_j$ and a contact point $x_j$, set
\[
y_j:=\Psi_j^{-1}(x_j),
\qquad
U_j:=U_{R_N}\circ\Psi_j^{-1},
\]
and define the corresponding transported boundary flux by
\begin{equation}
Q_j:=-\partial_{\nu_g}U_j(x_j).
\label{eq:Qj-definition-final}
\end{equation}
Thus $Q_j$ belongs to the successive approximation scheme and is
associated with the particular map $\Psi_j$.
The chain rule for the Riemannian gradient gives
\begin{equation}
\nabla_g U_j(x_j)
=
\bigl(d\Psi_j^{-1}(x_j)\bigr)^*
\nabla_g U_{R_N}(y_j).
\label{eq:transported-gradient-final}
\end{equation}
Hence
\begin{equation}
Q_j
=
-g_{x_j}\!\left(
\bigl(d\Psi_j^{-1}(x_j)\bigr)^*
\nabla_g U_{R_N}(y_j),
\nu_g(x_j)
\right).
\label{eq:transported-flux-final}
\end{equation}

\begin{definition}[Asymptotically conformal contractions]
\label{def:asymptotically-conformal}
Let $(\Psi_j)$ be a sequence of Riemannian contractions and let
$x_j$ be the associated contact points. We say that $(\Psi_j)$ is
\emph{asymptotically conformal at the contact points} if there exist
numbers $\mu_j\in(0,1)$, isometries
\[
P_j:T_{y_j}M\longrightarrow T_{x_j}M,
\]
and numbers $\varepsilon_j\downarrow0$ such that
\begin{equation}
\mu_j\longrightarrow1,
\qquad
\|d\Psi_j(y_j)-\mu_jP_j\|_g
\leq\varepsilon_j,
\label{eq:first-jet-convergence}
\end{equation}
and
\begin{equation}
\|\nu_g(x_j)-P_j\nu_g(y_j)\|_g
\leq\varepsilon_j.
\label{eq:normal-compatibility}
\end{equation}
\end{definition}

The next lemma gives the precise consequence of this condition for the
normal flux.

\begin{lemma}[Asymptotic stability of the transported flux]
\label{lem:asymptotic-flux-stability}
Assume that $U_{R_N}$ is $C^1$ in a neighborhood of
$\partial\mathcal B_{R_N}$ and that
\[
-\partial_{\nu_g}U_{R_N}=q(R_N)
\qquad\text{on }\partial\mathcal B_{R_N}.
\]
If the contractions $(\Psi_j)$ are asymptotically conformal at the
contact points in the sense of Definition~\ref{def:asymptotically-conformal},
then
\begin{equation}
Q_j-q(R_N)\longrightarrow0.
\label{eq:Qj-to-qRN}
\end{equation}
\end{lemma}

\begin{proof}
By \eqref{eq:first-jet-convergence} and the invertibility of
$d\Psi_j(y_j)$,
\[
d\Psi_j^{-1}(x_j)
=\mu_j^{-1}P_j^{-1}+o(1),
\]
and consequently
\[
\bigl(d\Psi_j^{-1}(x_j)\bigr)^*
=\mu_j^{-1}P_j+o(1)
\]
in operator norm. Since $\mu_j\to1$, the gradients
$\nabla_gU_{R_N}(y_j)$ remain bounded and the normal compatibility
\eqref{eq:normal-compatibility} yields
\[
Q_j
=
-\partial_{\nu_g}U_{R_N}(y_j)+o(1).
\]
The reference flux is constant on $\partial\mathcal B_{R_N}$, so
\[
-\partial_{\nu_g}U_{R_N}(y_j)=q(R_N).
\]
Therefore
\[
Q_j=q(R_N)+o(1),
\]
which proves \eqref{eq:Qj-to-qRN}.
\end{proof}

\begin{remark}
The asymptotic conformality assumption is a sufficient geometric
condition, not a consequence of (RC). In particular, (RC) controls the
flux in the direction needed for the comparison argument, whereas
Definition~\ref{def:asymptotically-conformal} controls the distortion of
the contraction strongly enough to identify its limiting flux.
\end{remark}

\begin{proposition}[Riemannian comparison at a contact point]
\label{prop:riemannian-contact-comparison}
Let
\[
K\subset\Omega_1\subset\Omega_2
\]
be admissible domains satisfying the Riemannian contraction hypothesis
{\rm (RC)}. Let
\[
u_i:=u_{\Omega_i},
\qquad i=1,2,
\]
and let
\[
v=u_2\circ\Psi^{-1}
\]
be the transported state in $\Omega_\Psi\setminus K$, where
$\Psi$ and the contact point $x_0$ are given by {\rm (RC)}. Then
\[
-\partial_{\nu_g}u_1(x_0)
>
-\partial_{\nu_g}v(x_0).
\label{eq:hopf-contact-comparison}
\]
Moreover, {\rm (RC)} gives
\[
-\partial_{\nu_g}v(x_0)
\geq
\kappa_\Psi
\bigl(-\partial_{\nu_g}u_2(y_0)\bigr),
\qquad
y_0=\Psi^{-1}(x_0),
\label{eq:RC-flux-comparison}
\]
with $\kappa_\Psi>1$. Consequently,
\begin{equation}
c_{\Omega_1}>\kappa_\Psi c_{\Omega_2}>c_{\Omega_2}.
\label{eq:riemannian-c-monotonicity}
\end{equation}
\end{proposition}

\begin{proof}
By {\rm (RC)}, $v$ is a weak subsolution and satisfies
\[
v=1\quad\text{on }\partial K,
\qquad
v=0\quad\text{on }\partial\Omega_\Psi.
\]
The comparison principle gives
\[
v\leq u_1
\qquad\text{in }\Omega_\Psi\setminus K.
\]
At the contact point $x_0$,
\[
v(x_0)=u_1(x_0)=0.
\]
If $v\not\equiv u_1$, the strong comparison principle and Hopf's
boundary point lemma yield
\[
-\partial_{\nu_g}u_1(x_0)
>
-\partial_{\nu_g}v(x_0).
\]
The equality of the outward normals at the contact point, required in
{\rm (RC)}, permits both derivatives to be computed with respect to the
same normal.

Now let
\[
y_0=\Psi^{-1}(x_0).
\]
Since $\Omega_2$ is stationary,
\[
-\partial_{\nu_g}u_2(y_0)=c_{\Omega_2}.
\]
The boundary-flux part of {\rm (RC)} gives
\[
-\partial_{\nu_g}v(x_0)
\geq
\kappa_\Psi c_{\Omega_2}.
\]
Since
\[
-\partial_{\nu_g}u_1=c_{\Omega_1}
\qquad\text{on }\partial\Omega_1,
\]
we obtain
\[
c_{\Omega_1}>\kappa_\Psi c_{\Omega_2}>c_{\Omega_2}.
\]
\end{proof}

\begin{remark}
This proposition shows that the subsolution property and the flux
comparison are part of the Riemannian contraction hypothesis {\rm (RC)}.
They should not be re-derived from the mere fact that $\Psi$ is a
diffeomorphism. The spherical radial construction of
Section~\ref{sec:spherical-radial-contraction} provides a concrete model
in which the transport can be analyzed explicitly.
\end{remark}

\subsection{Transformation of the $p$-Laplacian under a Riemannian contraction}
\label{sec:transformation-p-laplacian}

We make precise the role of the contraction $\Psi_j$. A crucial
point is that a general diffeomorphism does not preserve the
$p$-Laplacian associated with the fixed metric $g$. If
\[
\Psi:U\longrightarrow V
\]
is a diffeomorphism and
\[
h:=\Psi^\ast g
\]
is the pullback metric on $U$, then the $p$-Laplacian is natural with
respect to the pair $(g,h)$, not with respect to $g$ on both sides.
This distinction is essential in the Riemannian adaptation.

Let $u$ be a smooth function on $U$ and set
\[
U^\Psi:=u\circ\Psi^{-1}
\qquad\text{on }V.
\]
For every $\varphi\in C_c^\infty(V)$, put
\[
\widehat\varphi:=\varphi\circ\Psi.
\]
Then the change-of-variables formula gives
\begin{align}
&\int_V
|\nabla_g U^\Psi|_g^{p-2}
\left\langle
\nabla_g U^\Psi,\nabla_g\varphi
\right\rangle_g\,dV_g
\nonumber\\
&\qquad=
\int_U
|\nabla_h u|_h^{p-2}
\left\langle
\nabla_h u,\nabla_h\widehat\varphi
\right\rangle_h\,dV_h.
\label{eq:pullback-p-laplace-weak}
\end{align}
Consequently,
\[
\Delta_{p,g}U^\Psi
=
0\quad\text{in }V
\]
is equivalent to
\[
\Delta_{p,h}u
=
0\quad\text{in }U,
\]
where
\[
h=\Psi^\ast g.
\]

In particular, even if
\[
\Delta_{p,g}u=0
\quad\text{in }U,
\]
one cannot conclude in general that
\[
\Delta_{p,g}U^\Psi=0
\quad\text{in }V.
\]
Such a conclusion requires an additional compatibility condition between
$h=\Psi^\ast g$ and $g$. This is the precise analytical content that
must be imposed in the Riemannian contraction hypothesis (RC).

\begin{definition}[Riemannian $p$-harmonicity residual]
\label{def:riemannian-p-harmonic-residual}
For a diffeomorphism $\Psi:U\to V$, define the residual of the transported
function $U^\Psi=u\circ\Psi^{-1}$ by
\begin{align}
\mathcal R_{\Psi,g}[u](\widehat\varphi)
:={}&
\int_U
|\nabla_h u|_h^{p-2}
\left\langle
\nabla_hu,\nabla_h\widehat\varphi
\right\rangle_h\,dV_h
\nonumber\\
&-
\int_U
|\nabla_g u|_g^{p-2}
\left\langle
\nabla_gu,\nabla_g\widehat\varphi
\right\rangle_g\,dV_g,
\label{eq:riemannian-p-residual}
\end{align}
where $h=\Psi^\ast g$.
If $u$ is $p$-harmonic with respect to $g$, then
\[
\mathcal R_{\Psi,g}[u]
\]
measures exactly the failure of $U^\Psi$ to be $p$-harmonic with respect
to the fixed metric $g$.
\end{definition}

\begin{assumption}[Riemannian contraction compatibility]
\label{ass:RC-p-laplace}
We assume that the contractions $\Psi_j$ satisfy a uniform
sign condition on the residual in the comparison region, strong enough
to imply that the transported reference state $U_j$ is a weak
subsolution (respectively supersolution, according to the orientation
of the comparison) of the $p$-Laplace equation with respect to $g$.
We additionally assume that the residual tends to zero in the
asymptotic regime:
\[
\mathcal R_{\Psi_j,g}[u_{R_N}]
\longrightarrow0
\]
in the dual of the relevant test-function space.
\end{assumption}

\begin{remark}
Condition \textnormal{(RC)} should therefore not be interpreted merely
as a bound on $d\Psi_j$. The weak identity
\eqref{eq:pullback-p-laplace-weak} shows that the relevant object is the
difference between the $p$-energy tensors associated with $g$ and
$\Psi_j^\ast g$. Bounds on the first jet of $\Psi_j$ can be used to
verify this condition in concrete geometries, but they do not by
themselves imply the required sign condition.
\end{remark}

\subsection{Example: A concrete spherical model: explicit radial construction and contraction}
\label{sec:spherical-radial-contraction}

We give a completely explicit example on a curved manifold. The goal is
not only to display a radial solution, but to carry out, step by step, all
the constructions entering the abstract framework: the reference domains,
the $p$-harmonic states, the Bernoulli constant, the monotonicity with
respect to the outer radius, the contraction map, the transported state,
and the exact transformation of the boundary flux. This example also makes
clear which part of the abstract contraction hypothesis~{\rm (RC)} is
verified exactly in the spherical model and which part remains an additional
comparison assumption in the general theorem.

\medskip
\noindent
\textbf{Step 1: the geometric setting.}

We work on the round sphere of radius $\mathcal R>0$,
\[
(M,g)=\bigl(S^n_{\mathcal R},g_{\mathcal R}\bigr),
\qquad n\ge2,
\]
and fix a centre $x_0\in S^n_{\mathcal R}$. Let
\[
0<\rho<\frac{\pi\mathcal R}{2}
\]
and take as fixed inner obstacle the closed geodesic ball
\[
K=\overline{B_g(x_0,\rho)}.
\]
For every radius
\[
L\in\left(\rho,\frac{\pi\mathcal R}{2}\right),
\]
we define the admissible outer domain
\[
\Omega_L:=B_g(x_0,L).
\]
Thus $L$ is, by definition, the \emph{outer geodesic radius} of the
domain $\Omega_L$, while $\rho$ is the fixed geodesic radius of the
inner obstacle $K$. In particular,
\[
K\subset\Omega_L,
\qquad
\partial K=\{r=\rho\},
\qquad
\partial\Omega_L=\{r=L\}.
\]
The restriction $L<\pi\mathcal R/2$ is convenient because the area density
of geodesic spheres is then strictly increasing.

\medskip
\noindent
\textbf{Step 2: geodesic polar coordinates and the radial density.}

In geodesic polar coordinates $(r,\theta)$ centred at $x_0$, the round
metric is
\[
g_{\mathcal R}
=dr^2+\mathcal R^2\sin^2(r/\mathcal R)
 g_{\mathbb S^{n-1}}.
\]
The Riemannian volume element is therefore
\[
dV_{g_{\mathcal R}}
=J_{\mathcal R}(r)\,dr\,d\theta,
\qquad
J_{\mathcal R}(r)
:=\mathcal R^{n-1}\sin^{n-1}(r/\mathcal R).
\]
The function $J_{\mathcal R}$ is positive on $(0,\pi\mathcal R)$ and,
for $0<r<\pi\mathcal R/2$,
\[
J_{\mathcal R}'(r)
=(n-1)\mathcal R^{n-2}
\sin^{n-2}(r/\mathcal R)\cos(r/\mathcal R)>0.
\]
Hence $J_{\mathcal R}$ is strictly increasing on the whole interval in
which our reference family is considered.

\medskip
\noindent
\textbf{Step 3: explicit radial $p$-harmonic state.}

Fix $L\in(\rho,\pi\mathcal R/2)$. We consider
\[
\begin{cases}
-\Delta_{p,g_{\mathcal R}}u_L=0
&\text{in }\Omega_L\setminus K,\\
 u_L=1&\text{on }\partial K,\\
 u_L=0&\text{on }\partial\Omega_L.
\end{cases}
\tag{S$_L$}
\]
By rotational symmetry, we look for a radial solution
\[
u_L(x)=\varphi_L(r).
\]
For a radial function $u(x)=\varphi(r)$, one has
\[
\Delta_{p,g_{\mathcal R}}u
=
\frac1{J_{\mathcal R}(r)}
\frac{d}{dr}
\left(
J_{\mathcal R}(r)|\varphi'(r)|^{p-2}\varphi'(r)
\right).
\]
Consequently, equation (S$_L$) is equivalent to
\[
\frac{d}{dr}
\left(
J_{\mathcal R}(r)|\varphi_L'(r)|^{p-2}\varphi_L'(r)
\right)=0.
\]
Therefore there is a constant $C_L^{(0)}$ such that
\[
J_{\mathcal R}(r)|\varphi_L'(r)|^{p-2}\varphi_L'(r)=C_L^{(0)}.
\]
Because the boundary values are $1$ at $r=\rho$ and $0$ at $r=L$,
$\varphi_L$ is decreasing and hence $\varphi_L'<0$. We may therefore
write
\[
-\varphi_L'(r)
=C_LJ_{\mathcal R}(r)^{-1/(p-1)},
\]
where $C_L>0$. Integrating from $\rho$ to $r$ gives
\[
\varphi_L(r)
=1-C_L\int_\rho^r
J_{\mathcal R}(s)^{-1/(p-1)}\,ds.
\]
The condition $\varphi_L(L)=0$ determines $C_L$ uniquely:
\[
C_L
=\left(
\int_\rho^L
J_{\mathcal R}(s)^{-1/(p-1)}\,ds
\right)^{-1}.
\]
It is useful to introduce the notation
\[
I_{\mathcal R}(L)
:=\int_\rho^L
J_{\mathcal R}(s)^{-1/(p-1)}\,ds.
\]
Then
\[
C_L=I_{\mathcal R}(L)^{-1},
\qquad
\varphi_L(r)=1-\frac{I_{\mathcal R}(r)}{I_{\mathcal R}(L)},
\]
where
\[
I_{\mathcal R}(r):=
\int_\rho^rJ_{\mathcal R}(s)^{-1/(p-1)}\,ds.
\]
Thus the complete radial state is explicitly known.

\medskip
\noindent
\textbf{Step 4: the Bernoulli flux $q_{\mathcal R}(L)$.}

On the outer geodesic sphere $\partial\Omega_L=\{r=L\}$, the outward
unit normal is
\[
\nu_g=\partial_r.
\]
The Bernoulli quantity associated with $\Omega_L$ is therefore
\[
q_{\mathcal R}(L)
:=-\partial_{\nu_g}u_L\big|_{\partial\Omega_L}
=-\varphi_L'(L).
\]
Using the formula above,
\[
q_{\mathcal R}(L)
=
\frac{J_{\mathcal R}(L)^{-1/(p-1)}}
{I_{\mathcal R}(L)}
=
\frac{J_{\mathcal R}(L)^{-1/(p-1)}}
{\displaystyle\int_\rho^L
J_{\mathcal R}(s)^{-1/(p-1)}\,ds}.
\label{eq:spherical-reference-flux}
\]
This gives the precise meaning of the notation $q_{\mathcal R}(L)$: it is
\emph{the normal boundary flux of the radial $p$-harmonic state whose
outer boundary has geodesic radius $L$}.

Moreover, $q_{\mathcal R}$ is strictly decreasing. Indeed, setting
$N(L)=J_{\mathcal R}(L)^{-1/(p-1)}$, we have
$q_{\mathcal R}=N/I_{\mathcal R}$ and
\[
N'(L)
=-\frac1{p-1}J_{\mathcal R}'(L)
J_{\mathcal R}(L)^{-p/(p-1)}<0.
\]
Since $I_{\mathcal R}'(L)=N(L)$,
\[
q_{\mathcal R}'(L)
=\frac{N'(L)I_{\mathcal R}(L)-N(L)^2}
{I_{\mathcal R}(L)^2}<0.
\]
Hence
\[
q_{\mathcal R}'(L)<0
\quad\text{for }L\in(\rho,\pi\mathcal R/2).
\]
In particular, if $\rho<L'<L<\pi\mathcal R/2$, then
\[
q_{\mathcal R}(L')>q_{\mathcal R}(L).
\]
This is exactly the expected monotonicity: as the outer domain
contracted, the Bernoulli flux increases.

\medskip
\noindent
\textbf{Step 5: solving the Bernoulli condition in the reference family.}

Suppose that the prescribed Bernoulli constant is $c>0$. In the spherical
family, the free-boundary equation
\[
-\partial_{\nu_g}u_L=c
\quad\text{on }\partial\Omega_L
\]
reduces to the scalar equation
\[
q_{\mathcal R}(L)=c.
\]
Because $q_{\mathcal R}$ is strictly decreasing, there is at most one
radius $L$ solving this equation. Whenever $c$ belongs to the range of
$q_{\mathcal R}$ on the chosen interval, the corresponding radius is
therefore unique. This gives a completely explicit realization of the
reference-family mechanism in~{\rm (RA)}.

\medskip
\noindent
\textbf{Step 6: definition of the contracted outer radius $L'$.}

Let $L$ be the outer radius of the reference domain $\Omega_L$ introduced above.
We now choose a smaller outer radius $L'$ satisfying
\[
\rho<L'<L<\frac{\pi\mathcal R}{2}.
\]
The corresponding geodesic ball
\[
\Omega_{L'}=B_g(x_0,L')
\]
is the contracted domain, while the inner obstacle
\[
K=\overline{B_g(x_0,\rho)}
\]
remains unchanged. Hence
\[
K\subset\Omega_{L'}\subset\Omega_L.
\]
We use the following notation throughout the spherical construction:
\begin{center}
\begin{tabular}{c|l}
$L$ & original outer geodesic radius,\\
$L'$ & contracted outer geodesic radius,\\
$\Omega_L$ & original geodesic ball,\\
$\Omega_{L'}$ & contracted geodesic ball,\\
$G_{\mathcal R}$ & radial $p$-harmonic coordinate,\\
$A_{L,L'}$ & radial transport factor,\\
$f$ & radial transport function,\\
$\Psi_{L,L'}$ & radial contraction map,\\
$q_{\mathcal R}(L)$ & Bernoulli flux for $\Omega_L$,\\
$Q_{L,L'}$ & transported boundary flux,\\
$L'_{\varepsilon}$ & contracted radius in the limiting family,\\
$\Psi_{\varepsilon}$ & $\varepsilon$-dependent spherical contraction,\\
$Q_{\varepsilon}$ & flux associated with $\Psi_{\varepsilon}$,\\
$\kappa_{L,L'}$ & exact spherical flux amplification factor.
\end{tabular}
\end{center}
Thus
\[
\Omega_{L'}\subset\Omega_L,
\qquad
\partial\Omega_L=\{r=L\},
\qquad
\partial\Omega_{L'}=\{r=L'\}.
\]
The word ``contraction'' refers here to the inclusion of the domains.
It does not mean that the derivative of the radial coordinate map must be
smaller than one at every point.

\medskip
\noindent
\textbf{Step 7: construction of the radial transport map.}

We seek a diffeomorphism
\[
\Psi_{L,L'}:\Omega_L\setminus K
\longrightarrow\Omega_{L'}\setminus K
\]
of the form
\[
\Psi_{L,L'}(r,\theta)=(f(r),\theta),
\]
with
\[
f(\rho)=\rho,
\qquad
f(L)=L',
\qquad
f'(r)>0.
\]
The inner boundary is therefore fixed and the outer boundary is transported
from radius $L$ to radius $L'$.

Let
\[
U_{L,L'}:=u_L\circ\Psi_{L,L'}^{-1}
\]
be the transported state. If $s=f(r)$, then
\[
U_{L,L'}(s,\theta)=u_L(r,\theta)=\varphi_L(r)
\]
and hence
\[
U_{L,L'}'(s)=\frac{\varphi_L'(r)}{f'(r)}.
\]
We now determine $f$ so that $U_{L,L'}$ is again $p$-harmonic with
respect to the same spherical metric $g_{\mathcal R}$.

\medskip
\noindent
\textbf{Step 8: the differential equation for $f$.}

Since $u_L$ satisfies
\[
J_{\mathcal R}(r)|\varphi_L'(r)|^{p-2}\varphi_L'(r)
=C_L^{(0)},
\]
we obtain
\[
J_{\mathcal R}(f(r))
|U_{L,L'}'(f(r))|^{p-2}U_{L,L'}'(f(r))
=
C_L^{(0)}
\frac{J_{\mathcal R}(f(r))}
{J_{\mathcal R}(r)[f'(r)]^{p-1}}.
\]
Therefore the transported radial flux is constant if and only if there
is a constant $A>0$ such that
\[
\frac{J_{\mathcal R}(f(r))}
{J_{\mathcal R}(r)[f'(r)]^{p-1}}
=A^{-(p-1)}.
\label{eq:spherical-radial-compatibility}
\]
Equivalently,
\[
f'(r)
=A\left(
\frac{J_{\mathcal R}(f(r))}
{J_{\mathcal R}(r)}
\right)^{1/(p-1)}.
\]

Define
\[
G_{\mathcal R}(r)
:=\int_\rho^rJ_{\mathcal R}(s)^{-1/(p-1)}\,ds.
\]
Then
\[
G_{\mathcal R}'(r)=J_{\mathcal R}(r)^{-1/(p-1)}>0,
\]
so $G_{\mathcal R}$ is strictly increasing and admits a $C^1$ inverse
on the interval under consideration. Dividing the differential equation
for $f$ by $J_{\mathcal R}(f(r))^{1/(p-1)}$ gives
\[
\frac{f'(r)}{J_{\mathcal R}(f(r))^{1/(p-1)}}
=
A\,J_{\mathcal R}(r)^{-1/(p-1)}.
\]
Thus
\[
\frac{d}{dr}G_{\mathcal R}(f(r))
=A\frac{d}{dr}G_{\mathcal R}(r),
\]
and therefore
\[
G_{\mathcal R}(f(r))=AG_{\mathcal R}(r)+B.
\]
The condition $f(\rho)=\rho$ gives $B=0$. The condition $f(L)=L'$ gives
\[
A=A_{L,L'}
:=\frac{G_{\mathcal R}(L')}{G_{\mathcal R}(L)}.
\]
Since $L'<L$ and $G_{\mathcal R}$ is strictly increasing,
\[
0<A_{L,L'}<1.
\]
We have therefore obtained the explicit formula
\[
G_{\mathcal R}(f(r))
=\frac{G_{\mathcal R}(L')}{G_{\mathcal R}(L)}
G_{\mathcal R}(r).
\label{eq:spherical-conformal-radial-map}
\]
Equivalently,
\[
 f(r)=G_{\mathcal R}^{-1}
 \left(
 \frac{G_{\mathcal R}(L')}{G_{\mathcal R}(L)}
 G_{\mathcal R}(r)
 \right).
\]
This formula defines uniquely the required increasing radial transport.
Differentiating it gives
\[
 f'(r)=A_{L,L'}
 \left(
 \frac{J_{\mathcal R}(f(r))}{J_{\mathcal R}(r)}
 \right)^{1/(p-1)}.
\label{eq:spherical-f-derivative}
\]

\medskip
\noindent
\textbf{Step 9: verification that the transported state is $p$-harmonic.}

Substituting the last formula into
\eqref{eq:spherical-radial-compatibility}, we find
\[
\frac{J_{\mathcal R}(f(r))}
{J_{\mathcal R}(r)[f'(r)]^{p-1}}
=A_{L,L'}^{-(p-1)},
\]
which is independent of $r$. Hence
\[
\frac{d}{ds}
\left(
J_{\mathcal R}(s)
|U_{L,L'}'(s)|^{p-2}U_{L,L'}'(s)
\right)=0.
\]
Therefore
\[
-\Delta_{p,g_{\mathcal R}}U_{L,L'}=0
\quad\text{in }\Omega_{L'}\setminus K.
\]
The boundary values are also preserved:
\[
U_{L,L'}=1\quad\text{on }\partial K,
\qquad
U_{L,L'}=0\quad\text{on }\partial\Omega_{L'}.
\]
Thus $U_{L,L'}$ solves exactly the same Dirichlet problem as the radial
state $u_{L'}$ on the smaller annulus. By uniqueness of the Dirichlet
problem,
\[
U_{L,L'}=u_{L'}.
\]
This identity is important: the radial transport has not produced an
arbitrary $p$-harmonic function; it has transported $u_L$ exactly onto the
canonical radial state corresponding to the contracted radius $L'$.

\medskip
\noindent
\textbf{Step 10: the transported boundary flux $Q_{L,L'}$.}

We now define the quantity $Q_{L,L'}$ appearing in the comparison.
It is the geometric normal derivative of the transported state at the new
outer boundary:
\[
Q_{L,L'}
:=-\partial_{\nu_g}U_{L,L'}
\big|_{\partial\Omega_{L'}}.
\]
Since $\nu_g=\partial_r$ on $\partial\Omega_{L'}$, this is simply
\[
Q_{L,L'}=-U_{L,L'}'(L').
\]
Because $L'=f(L)$,
\[
U_{L,L'}'(L')
=\frac{\varphi_L'(L)}{f'(L)},
\]
and consequently
\[
Q_{L,L'}
=\frac{q_{\mathcal R}(L)}{f'(L)}.
\label{eq:Q-over-fprime}
\]
Using the explicit expression for $f'(L)$,
\[
Q_{L,L'}
=
\frac{q_{\mathcal R}(L)}
{A_{L,L'}
\left(J_{\mathcal R}(L')/J_{\mathcal R}(L)\right)^{1/(p-1)}}.
\label{eq:exact-boundary-flux-transport}
\]

There is, however, an even more informative identity. Since we proved in
Step 9 that $U_{L,L'}=u_{L'}$, its boundary flux must equal the reference
flux at radius $L'$. Hence
\[
Q_{L,L'}=q_{\mathcal R}(L').
\label{eq:Q-equals-qLprime}
\]
Combining this with~\eqref{eq:Q-over-fprime} gives the exact identity
\[
\frac{q_{\mathcal R}(L')}{q_{\mathcal R}(L)}
=\frac1{f'(L)}.
\label{eq:exact-amplification-ratio}
\]
Since $L'<L$ and $q_{\mathcal R}$ is strictly decreasing,
\[
q_{\mathcal R}(L')>q_{\mathcal R}(L),
\]
and therefore
\[
Q_{L,L'}>q_{\mathcal R}(L).
\]
Thus the spherical contraction gives an \emph{exact strict flux
amplification}. If we define
\[
\kappa_{L,L'}
:=\frac{Q_{L,L'}}{q_{\mathcal R}(L)}
=\frac{q_{\mathcal R}(L')}{q_{\mathcal R}(L)}>1,
\]
then
\[
Q_{L,L'}=\kappa_{L,L'}q_{\mathcal R}(L).
\]
This is the spherical analogue of the amplification factor
$\kappa_\Psi>1$ appearing in~{\rm (RC)}.

\medskip
\noindent
\textbf{Step 12: disappearance of the contraction and flux stability.}

Finally, let
\[
L'_\varepsilon<L, \  0< \varepsilon << 1 \, \,  \mbox{such  that}\, \,
\qquad
L'_\varepsilon\longrightarrow L, \, \,  \mbox{when}\, \,  \varepsilon \rightarrow 0,
\]
and let
\[
\Psi_\varepsilon:=\Psi_{L,L'_\varepsilon}:\Omega_L
\longrightarrow\Omega_{L'_\varepsilon}
\]
denote the corresponding one-parameter family of diffeomorphic
contractions. We use the parameter $\varepsilon$ here deliberately:
$\Psi_\varepsilon$ is not one of the discrete maps $\Psi_j$ of the
successive approximation scheme. Let
\[
f_\varepsilon:=f_{L,L'_\varepsilon},
\qquad
Q_\varepsilon:=Q_{L,L'_\varepsilon}
=-\partial_{\nu_g}
\left(u_L\circ\Psi_\varepsilon^{-1}\right)
\big|_{\partial\Omega_{L'_\varepsilon}}.
\]
Thus $Q_\varepsilon$ denotes the flux associated with the
$\varepsilon$-dependent spherical contraction, whereas $Q_j$ denotes
the flux associated with the $j$-th contraction in the abstract
successive scheme.
Since $G_{\mathcal R}$ and $J_{\mathcal R}$ are continuous, when $\varepsilon \rightarrow 0,$
\[
A_\varepsilon
:=\frac{G_{\mathcal R}(L'_\varepsilon)}
{G_{\mathcal R}(L)}
\longrightarrow1,
\qquad
\frac{J_{\mathcal R}(L'_\varepsilon)}
{J_{\mathcal R}(L)}\longrightarrow1.
\]
Hence
\[
f_\varepsilon'(L)
=A_\varepsilon
\left(
\frac{J_{\mathcal R}(L'_\varepsilon)}
{J_{\mathcal R}(L)}
\right)^{1/(p-1)}
\longrightarrow1.
\]
Therefore
\[
Q_\varepsilon
=\frac{q_{\mathcal R}(L)}{f_\varepsilon'(L)}
\longrightarrow q_{\mathcal R}(L).
\]
Equivalently, because $Q_\varepsilon=q_{\mathcal R}(L'_\varepsilon)$,
this convergence is also an immediate consequence of the continuity of
$q_{\mathcal R}$:
\[
Q_\varepsilon=q_{\mathcal R}(L'_\varepsilon)
\longrightarrow q_{\mathcal R}(L).
\]
Moreover, the error is exactly
\[
|Q_\varepsilon-q_{\mathcal R}(L)|
=q_{\mathcal R}(L)
\left|\frac1{f_\varepsilon'(L)}-1\right|.
\label{eq:spherical-flux-error}
\]
Thus the strict amplification disappears continuously as the contraction
$L'_\varepsilon<L$ tends to the identity $L'=L$.

\medskip
\noindent
\textbf{Conclusion of the example.}

The spherical family gives a fully explicit realization of the abstract framework.

In particular,
\[
Q_{L,L'}=q_{\mathcal R}(L')
>q_{\mathcal R}(L),
\]
while
\[
L'_\varepsilon\uparrow L
\quad\Longrightarrow\quad
Q_\varepsilon\downarrow q_{\mathcal R}(L).
\]
The example therefore gives an explicit curved-space model in which the
monotonicity of the Bernoulli constant, the radial contraction, the exact
transport of the $p$-harmonic state, the strict flux amplification, and the
limiting stability of the flux can all be obtained from direct calculations.

\begin{remark}
The hypotheses
\[
c_j>c,\qquad q(R_N)>c
\]
fix the side from which the two monotone constructions approach the
target level. The convergence
\[
c_j-c_{j+1}\to0
\]
is a consequence of the convergence $c_j\downarrow\ell$, and
\[
k_N\to\infty
\]
follows from $c_j>c$ for every finite $j$ together with
$q(R_N)\to c$. Thus these two facts are not additional assumptions in
the stopping theorem.
\end{remark}

\begin{theorem}[Monotonicity of the Bernoulli constant]
\label{thm:monotonicity-c-riemannian}
Under the hypotheses of Proposition~\ref{prop:riemannian-contact-comparison},
if
\[
K\subset\Omega_1\subset\Omega_2,
\]
then
\[
c_{\Omega_1}>c_{\Omega_2}.
\]
Thus, along a decreasing sequence of admissible domains
\[
\Omega_{j+1}\subsetneq\Omega_j,
\]
the associated Bernoulli constants satisfy
\[
c_{j+1}<c_j.
\]
\end{theorem}

\begin{proof}
Set
\[
u_1:=u_{\Omega_1},
\qquad
u_2:=u_{\Omega_2}.
\]
By the Riemannian contraction hypothesis {\rm (RC)}, there exist an
admissible domain $\Omega_\Psi\subset\Omega_1$, a $C^2$ diffeomorphism
\[
\Psi:
\overline{\Omega_2\setminus K}
\longrightarrow
\overline{\Omega_\Psi\setminus K},
\]
and a contact point
\[
x_0\in\partial\Omega_\Psi\cap\partial\Omega_1
\]
such that
\[
\nu_{\Omega_\Psi}(x_0)=\nu_{\Omega_1}(x_0).
\]
Define the transported state
\[
v:=u_2\circ\Psi^{-1}
\qquad\text{in }\Omega_\Psi\setminus K.
\]
The hypothesis {\rm (RC)} asserts that $v$ is a weak subsolution of the
Riemannian $p$-Laplace equation and that
\[
v=1\quad\text{on }\partial K,
\qquad
v=0\quad\text{on }\partial\Omega_\Psi.
\]

\medskip
\noindent\textbf{Step 1: comparison of $v$ with $u_1$.}
Since $\Omega_\Psi\subset\Omega_1$, the function $u_1$ is
$p$-harmonic in $\Omega_\Psi\setminus K$ and satisfies
\[
u_1=1\quad\text{on }\partial K,
\qquad
u_1=0\quad\text{on }\partial\Omega_1.
\]
In particular,
\[
u_1\geq0\quad\text{on }\partial\Omega_\Psi,
\]
because $\partial\Omega_\Psi$ lies inside $\Omega_1$ except at the
contact point $x_0$, where $u_1(x_0)=0$. Hence
\[
(v-u_1)^+=0
\quad\text{on }\partial(\Omega_\Psi\setminus K).
\]
Thus
\[
w:=(v-u_1)^+\in W_0^{1,p}(\Omega_\Psi\setminus K).
\]

Using $w$ as an admissible test function in the weak subsolution
inequality for $v$ and in the weak equation for $u_1$, and subtracting
the two relations, gives
\[
\int_{\{v>u_1\}}
\left\langle
|\nabla_gv|^{p-2}\nabla_gv
-
|\nabla_gu_1|^{p-2}\nabla_gu_1,
\nabla_gv-\nabla_gu_1
\right\rangle_g\,dV_g
\leq0.
\]
The map
\[
\xi\longmapsto|\xi|_g^{p-2}\xi
\]
is strictly monotone for $1<p<\infty$. Therefore the integrand is
nonnegative and vanishes only when
\[
\nabla_gv=\nabla_gu_1.
\]
It follows that
\[
\nabla_g(v-u_1)^+=0
\quad\text{a.e. in }\Omega_\Psi\setminus K.
\]
Since $(v-u_1)^+$ has zero trace, we conclude
\[
v\leq u_1
\qquad\text{in }\Omega_\Psi\setminus K.
\]

\medskip
\noindent\textbf{Step 2: strict boundary comparison at the contact point.}
At the contact point $x_0$,
\[
v(x_0)=0=u_1(x_0).
\]
If $v\not\equiv u_1$, the strong comparison principle gives
\[
v<u_1
\qquad\text{in the interior of }\Omega_\Psi\setminus K.
\]
Both functions vanish at $x_0$ on the common tangent boundary, and
{\rm (RC)} gives
\[
\nu_{\Omega_\Psi}(x_0)=\nu_{\Omega_1}(x_0).
\]
The boundary point lemma of Hopf, applied with this common outward
normal, therefore yields
\[
-\partial_{\nu_g}u_1(x_0)
>
-\partial_{\nu_g}v(x_0).
\]
Since $\Omega_1$ is an optimal domain, its overdetermined boundary
condition is
\[
-\partial_{\nu_g}u_1=c_{\Omega_1}
\qquad\text{on }\partial\Omega_1.
\]
Consequently,
\[
c_{\Omega_1}>
-\partial_{\nu_g}v(x_0).
\]

\medskip
\noindent\textbf{Step 3: use of the flux amplification in {\rm (RC)}.}
Let
\[
y_0:=\Psi^{-1}(x_0).
\]
Because $\Omega_2$ is also an optimal domain,
\[
-\partial_{\nu_g}u_2(y_0)=c_{\Omega_2}.
\]
The flux part of {\rm (RC)} gives a constant $\kappa_\Psi>1$ such that
\[
-\partial_{\nu_g}v(x_0)
\geq
\kappa_\Psi
\left(
-\partial_{\nu_g}u_2(y_0)
\right)
=
\kappa_\Psi c_{\Omega_2}.
\]
Combining this inequality with the strict Hopf comparison obtained in
Step 2 gives the key chain
\[
c_{\Omega_1}
>
-\partial_{\nu_g}v(x_0)
\geq
\kappa_\Psi c_{\Omega_2}
>
c_{\Omega_2}.
\]
In particular,
\[
c_{\Omega_1}>c_{\Omega_2}.
\]

\medskip
\noindent\textbf{Step 4: return to the Lagrange multipliers.}
For every optimal domain,
\[
\lambda_\Omega
=
-\frac{p-1}{p}c_\Omega^p.
\]
The function
\[
c\longmapsto-\frac{p-1}{p}c^p
\]
is strictly decreasing on $(0,\infty)$. Hence
\[
c_{\Omega_1}>c_{\Omega_2}
\quad\Longrightarrow\quad
\lambda_{\Omega_1}<\lambda_{\Omega_2}.
\]
Thus the two formulations are equivalent:
\[
c_{\Omega_1}>c_{\Omega_2}
\qquad\Longleftrightarrow\qquad
\lambda_{\Omega_1}<\lambda_{\Omega_2}.
\]

Finally, applying the result to two consecutive domains of the
successive approximation,
\[
\Omega_{j+1}\subsetneq\Omega_j,
\]
gives
\[
c_{j+1}<c_j,
\qquad
\lambda_{j+1}>\lambda_j.
\]
This is precisely the monotonicity needed in the subsequent stopping
argument.
\end{proof}

\subsection{Identification of the Limiting Bernoulli Level}
\label{sec:identification-limiting-level}

By Theorem~\ref{thm:monotonicity-c-riemannian}, the successive
Bernoulli constants associated with the decreasing domains satisfy
\begin{equation}
c_0>c_1>c_2>\cdots.
\label{eq:decreasing-bernoulliconstants}
\end{equation}

For the successive construction we assume that every finite level
remains strictly above the prescribed level:
\begin{equation}
c_j>c
\qquad\text{for every finite }j.
\label{eq:levels-above-c}
\end{equation}
Consequently, the monotone sequence admits a limit
\begin{equation}
c_j\downarrow\ell,
\qquad
\ell\geq c.
\label{eq:level-limit}
\end{equation}

The first-crossing argument must be formulated with reference fluxes
approaching $c$ from the same side. We therefore choose a sequence of
reference radii $R_N$ such that
\begin{equation}
q(R_N)>c,
\qquad
q(R_N)\longrightarrow c.
\label{eq:qRN-from-above}
\end{equation}
For each $N$, define
\begin{equation}
k_N
=
\min\{j\geq0:\ c_j\leq q(R_N)\}.
\label{eq:riemannian-stopping-index-final}
\end{equation}
Whenever such an index exists, its minimality gives
\begin{equation}
c_{k_N}\leq q(R_N)<c_{k_N-1}.
\label{eq:first-crossing-bracketing}
\end{equation}

If $q(R_N)$ lies strictly between two consecutive levels, the
inequalities are strict. The weak form above is the one that remains
valid in the presence of equality.

The first-crossing argument requires in addition that
\begin{equation}
c_j-c_{j+1}\longrightarrow0.
\label{eq:successive-level-gap-final}
\end{equation}
Then \eqref{eq:first-crossing-bracketing} implies
\begin{equation}
0\leq q(R_N)-c_{k_N}
\leq
c_{k_N-1}-c_{k_N}
\longrightarrow0,
\label{eq:stopping-gap}
\end{equation}
provided
\[
k_N\longrightarrow\infty.
\]
Hence
\begin{equation}
q(R_N)-c_{k_N}\longrightarrow0.
\label{eq:qRN-cK}
\end{equation}
Since $q(R_N)\to c$, we obtain
\begin{equation}
c_{k_N}\longrightarrow c.
\label{eq:stopping-level-to-c}
\end{equation}
But $(c_{k_N})$ is a subsequence of the decreasing sequence
$(c_j)$, so
\[
c_{k_N}\longrightarrow\ell.
\]
Therefore
\begin{equation}
\ell=c.
\label{eq:ell-equals-c}
\end{equation}

The limiting boundary condition then follows from the boundary-flux
stability hypothesis (SC):
\begin{equation}
-\partial_{\nu_g}u_{\Omega^\ast}
=
c
\qquad\text{on }\partial\Omega^\ast.
\label{eq:exact-riemannian-bernoulli-level}
\end{equation}

\begin{theorem}[Riemannian stopping criterion]
\label{thm:riemannian-stopping}
Assume the compactness, state-stability and boundary-flux stability
hypotheses, together with (RC), (SC), and (RA). Suppose that the
successive construction produces
\[
c_0>c_1>c_2>\cdots>c
\]
and
\[
c_j\downarrow\ell.
\]
Let $(R_N)$ satisfy
\[
q(R_N)>c,
\qquad
q(R_N)\longrightarrow c.
\]
Assume, in addition, that for every sufficiently large $N$ the
first-crossing index
\[
k_N=\min\{j\geq0:c_j\leq q(R_N)\}
\]
exists. Then
\[
q(R_N)-c_{k_N}\longrightarrow0,
\qquad
c_{k_N}\longrightarrow c,
\]
and consequently
\[
\ell=c.
\]
The limiting pair $(\Omega^\ast,u_{\Omega^\ast})$ therefore satisfies
\[
\begin{cases}
-\Delta_{p,g}u_{\Omega^\ast}=0
&\text{in }\Omega^\ast\setminus K,\\
u_{\Omega^\ast}=1
&\text{on }\partial K,\\
u_{\Omega^\ast}=0
&\text{on }\partial\Omega^\ast,\\
-\displaystyle\frac{\partial u_{\Omega^\ast}}{\partial\nu_g}=c
&\text{on }\partial\Omega^\ast.
\end{cases}
\]
\end{theorem}

\begin{remark}
The existence of the first-crossing index is an explicit hypothesis of
this stopping theorem. It is not a consequence of monotonicity alone:
if the limit $\ell$ were strictly larger than $c$, then
$q(R_N)\to c$ could eventually remain below every $c_j$. Thus, in a
fully unconditional existence theorem, the first-crossing property
must be established from the geometry of the chosen reference family
and the successive construction.
\end{remark}

\begin{theorem}[Sufficient condition for a classical Riemannian Bernoulli solution]
\label{thm:final-riemannian-existence}
Let $(M,g)$ be a smooth Riemannian manifold and let $K\subset M$ be a
compact domain with sufficiently regular boundary. Assume that the
admissible class satisfies the compactness and regularity hypotheses
used in the existence theorem for the shape optimization problem.
Assume moreover that:

\begin{enumerate}[label={\rm (\roman*)}]
\item the optimal domains satisfy the Riemannian contraction
      hypothesis {\rm (RC)} and the associated monotonicity result;

\item the successive construction admits a nested sequence
      \[
      \Omega_0\supsetneq\Omega_1\supsetneq\cdots
      \]
      with
      \[
      c_0>c_1>\cdots>c;
      \]

\item the boundary and state stability condition {\rm (SC)} holds;

\item the reference family satisfies {\rm (RA)}, and there exists a
      sequence $R_N$ such that
      \[
      q(R_N)>c,\qquad q(R_N)\to c,
      \]
      for which the first-crossing indices are defined;

\item the limiting boundary $\partial\Omega^\ast$ is sufficiently
      regular for the classical normal derivative to be defined.
\end{enumerate}

Then the successive construction has a limit $\Omega^\ast$ and a
state $u^\ast$ satisfying
\[
\begin{cases}
-\Delta_{p,g}u^\ast=0
&\text{in }\Omega^\ast\setminus K,\\
u^\ast=1
&\text{on }\partial K,\\
u^\ast=0
&\text{on }\partial\Omega^\ast,\\
-\partial_{\nu_g}u^\ast=c
&\text{on }\partial\Omega^\ast.
\end{cases}
\]
In particular, $(\Omega^\ast,u^\ast)$ is a classical solution of the
Riemannian exterior Bernoulli problem.
\end{theorem}

\begin{proof}
The compactness and stability assumptions yield a limiting admissible
domain $\Omega^\ast$ and the corresponding state $u^\ast$. The
monotonicity theorem gives the decreasing sequence of Bernoulli
constants. The Riemannian stopping criterion,
Theorem~\ref{thm:riemannian-stopping}, gives $\ell=c$. Finally, (SC)
allows passage to the limit in the boundary normal derivative and
yields
\[
-\partial_{\nu_g}u^\ast=c
\qquad\text{on }\partial\Omega^\ast.
\]
The remaining equations and boundary values follow from the state
problem and stability.
\end{proof}

\begin{corollary}[Conditional uniqueness]
\label{cor:conditional-uniqueness-riemannian}
Assume the hypotheses of Theorem~\ref{thm:final-riemannian-existence} and
suppose that two classical solutions in the admissible class are obtained
through the same shape-optimization and contraction framework. Then their
Bernoulli constants coincide by the prescribed value $c$, and the strict
monotonicity result prevents two distinct nested admissible optimizers from
having the same level. Hence the solution is unique within this class.
\end{corollary}
\textbf{Conflicts of interests.} The authors declare that there is no conflict of interests.

\end{document}